\documentclass[11pt]{amsart}
\usepackage[T1]{fontenc}

\newcommand{\A}{\mathbb{A}}

\newcommand{\NN}{\mathbb{N}}
\newcommand{\PP}{\mathbb{P}}
\newcommand{\QQ}{\mathbb{Q}}

\newcommand{\ZZ}{\mathbb{Z}}

\newcommand{\cA}{\mathcal{A}}
\newcommand{\cB}{\mathcal{B}}

\newcommand{\ccD}{\mathcal{D}}

\newcommand{\cF}{\mathcal{F}}
\newcommand{\cG}{\mathcal{G}}

\newcommand{\cN}{\mathcal{N}}

\newcommand{\cS}{\mathcal{S}}
\newcommand{\cT}{\mathcal{T}}
\newcommand{\cU}{\mathcal{U}}
\newcommand{\cV}{\mathcal{V}}

\newcommand{\cZ}{\mathcal{Z}}

\newcommand{\fb}{\mathfrak{b}}

\newcommand{\fsl}{\mathfrak{sl}}

\newcommand{\dact}{\boldsymbol{.}}
\newcommand{\lra}{\longrightarrow}

\DeclareMathOperator{\Char}{char}

\DeclareMathOperator{\cx}{cx}

\DeclareMathOperator{\End}{End}
\DeclareMathOperator{\Ext}{Ext}
\DeclareMathOperator{\GL}{GL}

\DeclareMathOperator{\Hom}{Hom}

\DeclareMathOperator{\im}{im}
\DeclareMathOperator{\id}{id}
\DeclareMathOperator{\ind}{ind}

\DeclareMathOperator{\Mat}{Mat}
\DeclareMathOperator{\modd}{mod}

\DeclareMathOperator{\res}{res}
\DeclareMathOperator{\rk}{rk}

\DeclareMathOperator{\SL}{SL}

\usepackage{pictex}
\usepackage{amscd}
\usepackage{latexsym}
\usepackage{amssymb}
\usepackage{eucal}
\usepackage{eufrak}
\usepackage{verbatim}

\numberwithin{equation}{section}

\newtheorem{Theorem}{Theorem}[section]
\newtheorem{Lemma}[Theorem]{Lemma}
\newtheorem{Corollary}[Theorem]{Corollary}
\newtheorem{Proposition}[Theorem]{Proposition}

\theoremstyle{Theorem}
\newtheorem{Thm}{Theorem}[subsection]
\newtheorem{Lem}[Thm]{Lemma}
\newtheorem{Prop}[Thm]{Proposition}
\newtheorem{Cor}[Thm]{Corollary}

\newtheorem*{thm*}{Theorem}

\theoremstyle{remark}
\newtheorem*{Remark}{Remark}
\newtheorem*{Remarks}{Remarks}
\newtheorem*{Definition}{Definition}

\newtheorem*{Examples}{Examples}

\numberwithin{equation}{section}

\begin{document}

\title[Finite Group Schemes of Domestic Type]{Extensions of Tame Algebras and Finite Group Schemes of Domestic Representation Type}

\author{Rolf Farnsteiner}

\address
{Mathematisches Seminar, Christian-Albrechts-Universit\"at zu Kiel,
Ludewig-Meyn-Str.\ 4, 24098 Kiel, Germany.}
\email{rolf@math.uni-kiel.de}
\thanks{Supported by the D.F.G. priority program SPP1388 `Darstellungstheorie'.}
\date{\today}

\makeatletter
\makeatother

\subjclass[2010]{Primary 16G60, 14L15}

\begin{abstract} Let $k$ be an algebraically closed field. Given an extension $A\!:\!B$ of finite-dimensional $k$-algebras, we establish criteria ensuring that the representation-theoretic 
notion of polynomial growth is preserved under ascent and descent. These results are then used to show that principal blocks of finite group schemes of odd characteristic are of polynomial 
growth if and only if they are Morita equivalent to trivial extensions of radical square zero tame hereditary algebras. In that case, the blocks are of domestic representation type and the underlying 
group schemes are closely related to binary polyhedral group schemes. \end{abstract}

\maketitle

\section*{Introduction}
Let $k$ be an algebraically closed field. Given a finite-dimensional $k$-algebra $A$, we denote by $\modd_A$ the category of finitely generated left $A$-modules. We will write $[M]$ for the isoclass of $M \in \modd_A$. If $\modd_A$ affords only finitely many isoclasses of indecomposable objects, then $A$ is called {\it representation-finite}. A fundamental result of Drozd \cite{Dr} asserts that $A$ is
 either tame or wild in the sense of the following definition:

\bigskip

\begin{Definition} An algebra $A$ is referred to as {\it tame}, if for every $d >0$, there exist $(A,k[T])$-bimodules $X_1, \ldots, X_{n(d)}$ which are free $k[T]$-modules of rank $d$,  
such that all but finitely many isoclasses of indecomposable $A$-modules of dimension $d$ are of the form $[X_i\!\otimes_{k[T]}\!k_\lambda]$ for  $i \in \{1,\ldots, n(d)\}$ and 
$\lambda : k[T] \lra k$. 

We say that $A$ is {\it wild}, if there exists an $(A,k\langle x,y\rangle)$-bimodule $X$, which is a finitely generated free right $k\langle x,y\rangle$-module, such that the functor
\[ \modd k\langle x,y \rangle \lra \modd A \ \ ; \ \ M \mapsto X\!\otimes_{k\langle x,y\rangle}\!M\]
preserves indecomposables and reflects isomorphisms. \end{Definition}

\bigskip
\noindent
The notion of wildness derives from the fact that the module category a wild algebra $A$ is at least as complicated as that of any other algebra, rendering the classification of indecomposable $A$-modules a rather hopeless endeavor. 

\bigskip

\begin{Remark} In the definition of tameness, it suffices to require the parametrizing $(A,k[T])$-bimodules $(X_i)_{1\le i \le n(d)}$ of the $d$-dimensional indecomposable modules to be finitely
generated over $k[T]$. In that case, each torsion submodule $t(X_i)$ of the $k[T]$-module $X_i$ is an $(A,k[T])$-submodule, and there exists a non-zero polynomial $f_i \in k[T]$ such that $t(X_i)f_i = (0)$. Let $f:=f_1f_2\cdots f_{n(d)}$. Then the principal open subset $D(f) := \{ \lambda \in \Hom_{\rm alg}(k[T],k) \ ; \ \lambda(f) \ne 0\}$ is cofinite.

Let $M$ be a $d$-dimensional indecomposable $A$-module. Assume that, with only finitely many exceptions, we have $[M] = [X_i\!\otimes_{k[T]}\!k_\lambda]$ for some $i \in \{1,\ldots,n(d)\}$ and $\lambda: k[T] \lra k$. If $\lambda \in D(f)$, then $t(X_i)\!\otimes_{k[T]}\!k_\lambda = (0)$, and the right exactness of the tensor product implies that $[M] = [Y_i\!\otimes_{k[T]}\!k_\lambda]$, where $Y_i := X_i/t(X_i)$ is an $(A,k[T])$-bimodule that is a free $k[T]$-module of rank $d$. 

This implies in particular that tameness is preserved under Morita equivalence. In \cite{KZ} the authors show that this even holds for stable equivalence. \end{Remark}

\bigskip
\noindent
Let $A$ be a tame algebra. For each $d>0$, we let $\mu_A(d)$ be the minimum of all possible numbers $n(d)$ occurring in the definition. We thus obtain a function $\mu_A : \NN \lra \NN_0$, which we refer to as the {\it growth function} of the tame algebra $A$. The bimodules $X_1,\ldots, X_{n(d)}$ define morphisms
\[ f_{i,d} : \A^1 \lra \modd^d_A \ \ ; \ \ \lambda \mapsto X_i\!\otimes_{k[T]}\!k_\lambda,\]
with values in the variety $\modd^d_A$ of $d$-dimensional $A$-modules. Accordingly, these bimodules provide a continuous one-parameter families of indecomposable $A$-modules of dimension $d$. This fact notwithstanding, there exist prominent examples of tame algebras, where a classification of the indecomposables has remained elusive. This has led to the definition of various subclasses of tame algebras, whose module categories promise to be better behaved. In this paper, we are ultimately concerned with the following class, introduced by Ringel \cite{Ri}:  

\bigskip

\begin{Definition} An algebra $A$ is called {\it domestic}, if there exist $(A,k[T])$-bimodules $X_1,\ldots,X_m$ that are free of finite rank over $k[T]$ such that for every $d>0$, all but 
finitely many isoclasses of $d$-dimensional indecomposable $A$-modules are of the form $[X_i\!\otimes_{k[T]}\!(k[T]/(T\!-\!\lambda)^j)]$ for some $i \in \{1,\ldots, m\}, j \in \NN$ and $\lambda \in k$. \end{Definition}

\bigskip

\begin{Remarks} (1) \ A group algebra $kG$ of a finite group $G$ is representation-infinite and domestic if and only if $\Char(k)=2$, and the Sylow $2$-subgroups of $G$ are Klein four groups, cf.\ \cite[(4.4.4)]{Be}. The restricted enveloping algebra $U_0(\fsl(2))$ of the restricted Lie algebra $\fsl(2)$ over a field of characteristic $p>0$ is domestic. 

(2) By the above remark, domesticity is preserved under Morita equivalence. \end{Remarks}

\bigskip
\noindent
Finite groups and restricted Lie algebras are examples of finite group schemes. The purpose of this paper is determine those finite $k$-group schemes $\cG$ of characteristic $p\ge 3$, whose algebra $k\cG$ of measures has domestic representation type. To that end, we study particular algebra extensions $A\!:\!B$, and show that the more general notion of polynomial growth is preserved under ascent and descent.

Given a function $f : \NN \lra \NN_0$, we say that $f$ has {\it polynomial growth}, if there exist $c>0$ and $n\ge 0$ such that $f(d) \le cd^{n-1}$ for all $d\ge 1$. The minimal number $n\in \NN_0$ with this property is the {\it rate of growth} $\gamma_f$ of the function $f$. The following class of algebras was introduced by Skowro\'nski \cite{Sk1}:

\bigskip

\begin{Definition} A tame algebra $A$ is said to be of {\it polynomial growth}, if $\mu_A$ has polynomial growth. In that case, we refer to $\gamma_A := \gamma_{\mu_A}$ as the  {\it rate of growth} of the tame algebra $A$. \end{Definition}

\bigskip

\begin{Remarks} Let $A$ be tame. 
\begin{enumerate}
\item[(a)] In view of Brauer-Thrall II (cf.\ \cite[(IV.5)]{ASS}), the algebra $A$ is representation-finite if and only if $A$ is tame and $\gamma_A = 0$.
\item[(b)] Thanks to \cite[(5.7)]{CB}, the algebra $A$ is domestic if and only if $A$ is tame and $\gamma_A\le1$.
\item[(c)] Suppose that $A$ is of polynomial growth. If $B$ is a $k$-algebra that is Morita equivalent to $A$, then the above remark implies that $B$ is also of
polynomial growth. 
\item[(d)] If $A$ is of polynomial growth and $I\unlhd A$ is an ideal, then parametrizing $(A,k[T])$-bimodules $X_1,\ldots, X_{\mu_A(d)}$ induce parametrizing $(A/I,k[T])$-bimodules 
$(X_j/IX_j)_{1\le j \le \mu_A(d)}$, and the above remark implies that $A/I$ is of polynomial growth of rate $\gamma_{A/I}\le \gamma_A$.
\item[(e)] As will be shown in Section 4, group schemes of polynomial growth are of domestic representation type. \end{enumerate}  \end{Remarks} 

\bigskip
\noindent
Our main result provides a criterion guaranteeing that polynomial growth is preserved under passage between the constituents of certain extensions $A\!:\!B$. We say that an extension $A\!:\!B$ of $k$-algebras is {\it split}, if $B$ is a direct summand of the $(B,B)$-bimodule $A$. We refer to $A\!:\!B$ as {\it separable}, if the multiplication of $A$ induces a split surjective morphism $\mu : A\!\otimes_B\!A \lra A$ of $(A,A)$-bimodules. 

\bigskip

\begin{thm*} Suppose that $A\!:\!B$ is an extension of algebras.
\begin{enumerate} 
\item If $A\!:\!B$ is split and $A$ is of polynomial growth, then $B$ is of polynomial growth, and $\gamma_B \le \gamma_A\!+\!1$.
\item If $A\!:\!B$ is separable and $B$ is of polynomial growth, then $A$ is of polynomial growth and $\gamma_A \le \gamma_B\!+\!1$.
\end{enumerate}\end{thm*}

\bigskip
\noindent
It is well-known that finite representation type is preserved when descending via a split extension or ascending via a separable extension, cf.\ \cite{Pi}.

Our approach involves a geometric interpretation of growth functions that elaborates on a result by de la Pe\~na \cite{dP1}. Section $1$ provides the requisite tools which are then applied in Section $2$, where the above result and some refinements are established. The concluding Section presents examples and applications concerning smash products and finite group schemes. In particular, we show that the principal block $\cB_0(\cG)$ of a finite group scheme of characteristic $p\ge 3$
is representation-infinite and domestic if and only if $\cB_0(\cG)$ is Morita equivalent to a trivial extension of a radical square zero tame hereditary algebra. Accordingly, one has a fairly good understanding of the module categories of these blocks (cf.\ \cite[(V.3.2)]{Ha}). 

\bigskip
 
\section{Geometric Characterization of $\mu_A$}
Let $k$ be an algebraically closed field. Throughout, $A$ is assumed to be a finite-dimensional $k$-algebra. Given $n \in \NN$, we denote by $\modd_A^n$ the affine variety of $n$-dimensional $A$-modules. An element of $\modd_A^n$ is a homomorphism $\varrho : A \lra \End_k(k^n)$ of $k$-algebras. The general linear group $\GL_n(k)$ acts on $\modd^n_A$ via
\[ (g\dact \varrho)(a) := g\circ\varrho(a)\circ g^{-1} \ \ \ \ \ \ \forall \ g \in \GL_n(k), \, \varrho \in \modd^n_A, \, a \in A.\]
Note that the orbits under this action are just the isomorphism classes of the underlying $A$-modules. Consequently, the subset $\ind_A^n \subseteq \modd_A^n$
of $n$-dimensional indecomposable $A$-modules is $\GL_n(k)$-invariant.

A subset $C$ of a topological space $X$ is {\it locally closed}, if it is the intersection of an open and a closed subset of $X$. We say that $C$ is {\it constructible}, if it is a finite union
of locally closed sets. By virtue of \cite[(AG.1.3)]{Bo}, every constructible subset $C$ of a noetherian topological space $X$ contains a dense, open subset of its closure. In particular,
constructible subsets of varieties have this important property.

The following geometric criterion (cf.\ \cite[(1.2)]{dP1}), is based on Drozd's Tame-Wild dichotomy \cite{Dr}:

\bigskip

\begin{Proposition} \label{GD1} The following statements are equivalent:
\begin{enumerate}
\item The algebra $A$ is tame.
\item For every natural number $d\in \NN$ there exists a closed subset $C_d \subseteq \modd^d_A$ such that $\dim C_d \le 1$ and $\ind_A^d \subseteq \GL_d(k)\dact C_d$. \hfill $\square$\end{enumerate} \end{Proposition}

\bigskip
\noindent
Suppose that $A$ is tame with parametrizing bimodules $X_1,\ldots, X_{\mu_A(d)}$ for some $d>0$. Let $f_i : \A^1 \lra \modd^d_A$ be the morphism determined by $X_i$, and put $V_i := \overline{{\rm im}\,f_i}$. By definition, there exists a finite subset $F_d \subseteq \modd^d_A$ such that
\[ \ind^d_A \subseteq \bigcup_{i=1}^{\mu_A(d)} \GL_d(k)\dact V_i \cup \GL_d(k)\dact F_d.\]

\bigskip
\noindent
We record the following elementary observation:

\bigskip

\begin{Lemma} \label{GD2} Let $V$ be a one-dimensional, irreducible variety. If $C \subseteq V$ is constructible, then $C$ is either finite or cofinite. \end{Lemma}

\begin{proof} Since $C$ is constructible, it contains a dense open subset $O$ of its closure $\bar{C}$. 

If $\dim \bar{C} = 0$, then $C$ is finite. Otherwise, $\dim \bar{C} = 1$, so that $O$ is dense and open in $\bar{C} = V$. Thus, $V\!\smallsetminus\! O$ is a proper, closed subset of $V$, whence $\dim  V\!\smallsetminus\! O = 0$. As a result, $O$ and $C$ are cofinite subsets of $V$. \end{proof}

\bigskip

\begin{Lemma} \label{GD3} Let $d>0$. Suppose there exist one-dimensional closed, irreducible subsets $V_1,\ldots,V_{\ell(d)}$ of $\modd^d_A$, and a finite subset $F_d \subseteq \modd^d_A$ such that
\[ \ind^d_A \subseteq \bigcup_{i=1}^{\ell(d)}\GL_d(k)\dact V_i \cup \GL_d(k)\dact F_d,\]
with $\ell(d)$ being minimal subject to the above property. Then the following statements hold:
\begin{enumerate}
\item $|V_i \cap \GL_d(k)\dact x| < \infty \ \ \ \ \ \forall \ x \in \modd^d_A,\ 1 \le i \le \ell(d)$.
\item $V_i\cap \ind^d_A$ is cofinite in $V_i$ for every $i \in \{1, \ldots, \ell(d)\}$.
\item $U_i := V_i \cap [\ind^d_A\!\smallsetminus\! (\bigcup_{\ell \ne i} \GL_d(k)\dact V_\ell)]$ is a dense, open subset of $V_i$ for $1 \le i \le \ell(d)$. \end{enumerate}\end{Lemma}

\begin{proof}  (1) Let $x$ be an element of $\modd^d_A$. By Chevalley's theorem \cite[(10.20)]{GW}, $V_i \cap \GL_d(k)\dact x$ is a constructible subset of $V_i$, and therefore, by Lemma \ref{GD2}, finite or cofinite in $V_i$. If there exists an element $x_0 \in \modd^d_A$ with $V_i \cap \GL_d(k)\dact x_0$ cofinite in $V_i$, then there is a finite set
$H_d \subseteq \modd^d_A$ such that $V_i \subseteq \GL_d(k)\dact x_0\cup H_d$, so that $\GL_d(k)\dact V_i \subseteq \GL_d(k)\dact (H_d\cup\{x_0\})$. This
contradicts the choice of $\ell(d)$.

(2) Note that $\ind^d_A$ is constructible (cf.\ \cite[(1.1)]{dP1}). If the constructible subset $V_i \cap \ind^d_A \subseteq V_i$ is finite, then the $\GL_d(k)$-invariance of $\ind^d_A$ yields $\ind^d_A \subseteq \bigcup_{\ell \ne i} \GL_d(k)\dact V_\ell \cup \GL_d(k)\dact (F_d \cup (V_i \cap \ind^d_A))$, which contradicts the minimality of $\ell(d)$. Consequently,  $V_i \cap \ind^d_A \subseteq V_i$ is a cofinite subset of $V_i$.

(3) Let $i \in \{1, \ldots, \ell(d)\}$. By Chevalley's Theorem \cite[(10.20)]{GW}, the image of $\GL_d(k)\dact V_\ell$ of $V_\ell$ under the multiplication $\GL_d(k)\times\modd^d_A \lra \modd^d_A$ is constructible. As $\ind_A^d$ is constructible, it follows that $U_i$ is constructible. In view of Lemma \ref{GD2}, the set $U_i$ is finite, or cofinite in $V_i$.

By definition, we have
\[ \ind_A^d \subseteq \GL_d(k)\dact U_i \cup \bigcup_{\ell\ne i} \GL_d(k)\dact V_\ell \cup \GL_d(k)\dact F_d.\]
By choice of $\ell(d)$, the set $U_i$ is not finite. Hence $U_i$ is cofinite in $V_i$ and thus a non-empty, open subset of the one-dimensional irreducible variety $V_i$. This implies that $U_i$ lies dense in $V_i$.\end{proof}

\bigskip
\noindent
We require the following refinement of Proposition \ref{GD1}:

\bigskip
\begin{Proposition} \label{GD4} The following statements are equivalent:
\begin{enumerate}
\item There exists a function $\ell : \NN \lra \NN_0$ such that:
\begin{enumerate}
\item For every $d>0$ there exist one-dimensional irreducible closed subsets $V_1, \ldots, V_{\ell(d)} \subseteq \modd^d_A$, and a finite subset $F_d \subseteq \ind^d_A$ such that $\ind^d_A \subseteq \bigcup_{i=1}^{\ell(d)}\GL_d(k)\dact V_i \cup \GL_d(k)\dact F_d$, and 
\item $\ell(d)$ is minimal subject to the properties listed in (a). \end{enumerate}
\item The algebra $A$ is tame with growth function $\mu_A=\ell$. \end{enumerate}\end{Proposition}

\begin{proof} (1) $\Rightarrow$ (2) Thanks to Proposition \ref{GD1}, we know that $A$ is tame. Let $X_1, \ldots, X_{\mu_A(d)}$ be $(A,k[T])$-bimodules giving rise to morphisms
\[ f_i : \A^1 \lra \modd_A^d\]
such that $\ind_A^d \subseteq \GL_d(k)\dact (\bigcup_{i=1}^{\mu_A(d)}\im f_i)\cup \GL_d(k)\dact G_d$ for some finite set $G_d \subseteq \ind^d_A$. Condition (b) readily yields $\ell(d) \le \mu_A(d)$ for all $d>0$. 

Let $d>0$. By assumption, there exists a finite subset $F_d \subseteq \ind_A^d$ such that
\[ \ind_A^d \subseteq \bigcup_{i=1}^{\ell(d)}  \GL_d(k)\dact V_i \cup \GL_d(k)\dact F_d ,\]
with $\ell(d)$ being minimal subject to this property. For each $j \in \{1, \ldots, \ell(d)\}$ we have
\[ V_j\cap \ind_A^d \subseteq \bigcup_{i=1}^{\mu_A(d)} (V_j\cap \GL_d(k)\dact \im f_i) \cup (V_j\cap \GL_d(k)\dact G_d),\]
with each constituent of the union being constructible. By Lemma \ref{GD3}(1), the set $V_j\cap \GL_d(k)\dact G_d$ is finite, while Lemma \ref{GD3}(2) shows that $V_j = \overline{V_j\cap \ind_A^d}$. Hence there is an element $\zeta(j) \in \{1, \ldots, \mu_A(d)\}$ such that
\[ V_j = \overline{V_j\cap\GL_d(k)\dact \im f_{\zeta(j)}}.\]
In particular, the constructible set $V_j\cap \GL_d(k)\dact \im f_{\zeta(j)}$ is not finite. Thus, Lemma \ref{GD2} provides a finite set $R_j$ with $V_j =(V_j\cap\GL_d(k)\dact \im f_{\zeta(j)}) \cup R_j \subseteq  \GL_d(k)\dact \im f_{\zeta(j)} \cup R_j$, whence
\[ \GL_d(k)\dact V_j \subseteq \GL_d(k)\dact \im f_{\zeta(j)} \cup \GL_d(k) \dact R_j .\]
As a result, there is a function $\zeta : \{1,\ldots,\ell(d)\} \lra \{1,\ldots,\mu_A(d)\}$ such that
\[ \ind_A^d  \subseteq \bigcup_{j=1}^{\ell(d)} \GL_d(k)\dact \im f_{\zeta(j)} \cup \GL_d(k)\dact (\bigcup_{j=1}^{\ell(d)} R_j\cup F_d).\]
Hence all but finitely many isoclasses of $d$-dimensional indecomposable $A$-modules are of the form $[X_{\zeta(j)}\!\otimes_{k[T]}\!k_\lambda]$ for $j \in \{1, \ldots, \ell(d)\}$ and $\lambda : k[T] \lra k$.  Consequently, $\mu_A(d) \le |\im \zeta| \le \ell(d)$, so that $\ell = \mu_A$ is the growth function of the tame algebra $A$. 

(2) $\Rightarrow$ (1) Let $d>0$ and consider the $(A,k[T])$-bimodules $X_1,\ldots,X_{\mu_A(d)}$ that are provided by the definition of tameness. If $f_i : \A^1
\lra \modd^d_A$ are the associated morphisms, then the family $V_i:= \overline{\im f_i} \ \ (1\le i \le \mu_A(d))$ together with some finite set $F_d \subseteq \modd^d_A$ satisfies condition (a).

Suppose we also have 
\[ \ind_A^d \subseteq \bigcup_{j=1}^{\ell(d)}\GL_d(k)\dact W_j \cup \GL_d(k)\dact F'_d,\]
with $F'_d$ finite and $W_j$ irreducible, closed and of dimension $1$. The above arguments provide a function $\xi : \{1,\ldots,\mu_A(d)\} \lra \{1,\ldots,\ell(d)\}$
such that
\[ V_i =  \overline{V_i\cap \GL_d(k)\dact W_{\xi(i)}}.\]
Consequently, $\xi$ is injective, so that $\mu_A(d) \le \ell(d)$.\end{proof}

\bigskip
\noindent
For future reference, we record the following basic result:

\bigskip

\begin{Lemma} \label{GD5} The map
\[ s_{m,n} : \modd^m_A \times \modd^n_A \lra \modd^{m+n}_A \ \ ; \ \ (M,N) \mapsto M\oplus N\]
is a morphism of affine varieties. \hfill $\square$ \end{Lemma}

\bigskip

\section{Split Extensions}
In the sequel, we are going to study the behavior of polynomial growth under split extensions. Throughout this section, $A\!:\!B$ denotes an extension of finite dimensional $k$-algebras. 

We begin with the following technical subsidiary result:

\bigskip

\begin{Lemma}\label{SpE1} Suppose that $A\!:\!B$ is a split extension, with $A$ being of polynomial growth. Then there exists a function $\ell : \NN \lra \NN_0$ with the following properties:
\begin{enumerate}
\item For every $d>0$, there exist a finite subset $F_d \subseteq \ind_B^d$ and $(A,k[T])$-bimodules $X_1, \ldots , X_{\ell(d)}$ that are finitely generated free right $k[T]$-modules such that for every $V \in \ind_B^d\!\smallsetminus\!\GL_d(k)\dact F_d$ there exist $j \in \{1, \ldots, \ell(d)\}$ and $\lambda : k[T] \lra k$ with
\begin{enumerate}
\item[(a)] $V$ is isomorphic to a direct summand of $(X_j\!\otimes_{k[T]}\!k_\lambda)|_B$, and
\item[(b)] $X_j\!\otimes_{k[T]}\!k_\lambda$ is isomorphic to a direct summand of $A\!\otimes_B\!V$. \end{enumerate} 
\item The function $\ell$ has polynomial growth of rate $\gamma_\ell \le \gamma_A\!+\!1$.
\item If $\ell(d)$ is minimal subject to properties (a) and (b), then  
\[ |\{X_j\!\otimes_{k[T]}\!k_\lambda \ ; \lambda \in k\} \cap \GL_{r_j}(k)\dact x| < \infty\]
for all $j \in \{1,\ldots,\ell(d)\}$, $x \in \modd^{r_j}_A,\ r_j := \rk_{k[T]}(X_j)$. \end{enumerate}\end{Lemma}

\begin{proof} (1) If $M$ is an $(A,k[T])$-bimodule, we shall write $M(\lambda) := M\!\otimes_{k[T]}\!k_\lambda$ for ease of notation. Since $A$ is tame, we can find, for every $d>0$, a finite set $G_d \subseteq \ind_A^d$ and $(A,k[T])$-bimodules $M_1^d, \ldots M_{\mu_A(d)}^d$ of $k[T]$-rank $d$ such that every
$M \in \ind_A^d \!\smallsetminus\! \GL_d(k)\dact G_d$ is isomorphic to $M_i^d(\lambda)$ for some $i \in \{1, \ldots, \mu_A(d)\}$ and some $\lambda : k[T] \lra k$.

Consider the finite set $G_{(d)} := \bigcup_{i=d}^{(\dim_kA)d}G_i$ as well as the finite-dimensional $A$-module $Q := \bigoplus_{M \in G_{(d)}} M$.
By the Theorem of Krull-Remak-Schmidt, the set
\[ U_d := \{ V \in \ind_B^d \ ; \ V \ \text{is isomorphic to a direct summand of} \ Q|_B\}\]
is the union of finitely many $\GL_d(k)$-orbits. Hence there exists a finite set $F_d \subseteq \ind^d_B$ with
\[ U_d = \GL_d(k)\dact F_d.\]
Given $V \in \ind_B^d\!\smallsetminus\! \GL_d(k)\dact F_d$, the splitting property of $A\!:\!B$ provides a decomposition
\[ (A\!\otimes_B\!V)|_B \cong V \oplus W\]
of $B$-modules. There exist indecomposable $A$-modules $N_1, \ldots, N_q$ such that $A\!\otimes_B\!V \cong N_1\oplus \cdots \oplus N_q$, and the Theorem of Krull-Remak-Schmidt
provides $j \in \{1,\ldots, q\}$ such that $V$ is isomorphic to a direct summand of $N_j|_B$. Since $\dim_k A\!\otimes_B\!V \le (\dim_k A)d$, we have $N_j  \in \ind_A^t$ for some
$t \in \{ d, \ldots, (\dim_kA)d\}$. The assumption $N_j \in \GL_t(k)\dact G_t$ entails $V \in U_d$, a contradiction. Consequently, $N_j$ is isomorphic to some $M_i^t(\lambda)$. As a result, the modules $M_1^d, \ldots, M_{\mu_A(d)}^d, M_1^{d+1}, \ldots,$ $ M_{\mu_A(d(\dim_kA))}^{d(\dim_kA)}$ have the requisite properties. 

(2) Given $d>0$, we write $n:=\dim_kA$ and put $\{M_r^s \ ; \ 1\le s \le nd, \, 1\le r \le \mu_A(s)\} =: \{X_1,\ldots, X_{\ell(d)}\}$. By assumption, there exists a constant $c'>0$ such that $\mu_A(d) \le c'd^{\gamma_A-1}$ for all $d\ge 1$. Part (1) implies
\[ \ell(d) \le \sum_{i=1}^{nd}\mu_A(i) \le c' \sum_{i=1}^{nd} i^{\gamma_A-1} \le c'nd (nd)^{\gamma_A-1} = c'n^{\gamma_A}d^{\gamma_A}.\]
Consequently, $\ell$ has polynomial growth of rate $\gamma_\ell \le \gamma_A\!+\!1$.

(3) If the assertion is false, then there exist $j \in \{1,\ldots,\ell(d)\}$ and $x_0 \in \modd_A^{r_j}$ such that the constructible set $\cV_j := \{X_j\!\otimes_{k[T]}\!k_\lambda \ ; \lambda \in k\} \cap \GL_{r_j}(k)\dact x_0$ is cofinite in $\{X_j\!\otimes_{k[T]}\!k_\lambda \ ; \lambda \in k\}$. We write $\{X_j\!\otimes_{k[T]}\!k_\lambda \ ; \lambda \in k\}  = \cV_j \cup G_{r_j}$
for some finite set $G_{r_j} \subseteq \modd^{r_j}_A$, so that
\[ \{X_j\!\otimes_{k[T]}\!k_\lambda \ ; \lambda \in k\}  \subseteq \GL_{r_j}(k)\dact (G_{r_j}\cup\{x_0\}).\]
For every $M \in G_{r_j}\cup\{x_0\}$, we fix a decomposition of $M|_B$ into indecomposables, and consider the finite subset $\cF_d \subseteq \ind_B^d$ of $d$-dimensional indecomposable constituents of $\bigoplus_{M  \in G_{r_j}\cup\{x_0\}}M|_B$.

Let $V$ be an element of $\ind^d_B\!\smallsetminus\! \GL_d(k)\dact (F_d\cup\cF_d)$. Since $V\not \in \GL_d(k)\dact F_d$, there exist $q \in \{1,\ldots, \ell(d)\}$ and $\lambda : k[T] \lra k$ such that 
\begin{enumerate}
\item[(a)] $V$ is a direct summand of some $(X_q\!\otimes_{k[T]}\!k_\lambda)|_B$, and
\item[(b)] $X_q\!\otimes_{k[T]}\!k_\lambda$ is a direct summand of $A\!\otimes_B\!V$. \end{enumerate}  
In view of (a), the assumption $q=j$ implies $V \in \GL_d(k)\dact \cF_d$, a contradiction. Thus, $q\ne j$, which contradicts the minimality of $\ell(d)$. \end{proof}

\bigskip
\noindent
Under additional assumptions, the growth rate of the function $\ell : \NN \lra \NN_0$ of Lemma \ref{SpE1} is bounded by $\gamma_A$. For instance, if $\gamma_A=0$, then $A$ is representation-finite
and $\gamma_\ell = 0$. The following result provides conditions that apply in the context of skew group algebras. 

Let $M$ be an $A$-module. Given an automorphism $\zeta \in {\rm Aut}(A)$, we denote by $M^{(\zeta)}$ the $A$-module with underlying $k$-space $M$ and action
\[ a\dact m := \zeta^{-1}(a).m \ \ \ \ \ \forall \ a \in A, \ m \in M.\]
Below, we shall apply this twisting operation to bimodules, which we shall consider modules over the enveloping algebra $A\!\otimes_k\!A^{\rm op}$. 

\bigskip

\begin{Corollary} \label{SpE2} Let $A\!:\!B$ be an extension of $k$-algebras such that $A \cong \bigoplus_{i=1}^r A_i$ is an isomorphism of $(B,B)$-modules with $A_i \cong B^{(\id_B\otimes \zeta_i)}$ for some $\zeta_i \in {\rm Aut}(B)$. Suppose that $A$ is of polynomial growth. Then there exists a function $\ell : \NN \lra \NN_0$ with the following properties:
\begin{enumerate}
\item For every $d>0$, there exist a finite subset $F_d \subseteq \ind_B^d$ and $(A,k[T])$-bimodules $X_1, \ldots , X_{\ell(d)}$ that are finitely generated free right $k[T]$-modules such that for every $V \in \ind_B^d\!\smallsetminus\!\GL_d(k)\dact F_d$ there exist $i \in \{1, \ldots, \ell(d)\}$ and $\lambda : k[T] \lra k$ such that
\begin{enumerate}
\item[(a)] $V$ is isomorphic to a direct summand of $(X_i\!\otimes_{k[T]}\!k_\lambda)|_B$, and
\item[(b)] $X_i\!\otimes_{k[T]}\!k_\lambda$ is a direct summand of $A\!\otimes_B\!V$. \end{enumerate} 
\item The function $\ell$ has  polynomial growth of rate $\gamma_\ell \le \gamma_A$. \end{enumerate}\end{Corollary}

\begin{proof} Let $V \in \ind^d_B$. Our current assumption implies that
\[ (A\!\otimes_B\! V)|_B \cong \bigoplus_{i=1}^r (A_i\!\otimes_B\!V)|_B \cong \bigoplus_{i=1}^r V^{(\zeta_i)}.\]
Returning to the proof of Lemma \ref{SpE1}, we write $A\!\otimes_B\! V = \bigoplus_{i=1}^q N_i$ as a sum of indecomposable modules and obtain that
$t_i := \dim_kN_i = a_{N_i}d$ for $1\le i \le q$, where $1\le a_{N_i} \le r$. The arguments of (\ref{SpE1}) now imply that the modules $M_1^d, \ldots, M_{\mu_A(d)}^d, M_1^{2d}, \ldots,
M^{2d}_{\mu_A(2d)}, \ldots, M_{\mu_A(rd)}^{rd}$ have the requisite properties. Consequently,
\[ \ell(d) \le \sum_{i=1}^r \mu_A(id) \le rc'(rd)^{\gamma_A-1} = r^{\gamma_A}c'd^{\gamma_A-1},\]
as desired. \end{proof} 

\bigskip
\noindent
Let $V$ be a variety. Given $x \in V$, we denote by $\dim_xV$ the {\it local dimension} of $V$ at $x$. By definition, $\dim_xV$ is the maximum dimension of all irreducible components
of $V$ containing $x$.   

\bigskip

\begin{Theorem} \label{SpE3} Suppose that $A\!:\!B$ is a split extension.
If $A$ is of polynomial growth, then $B$ is of polynomial growth with $\gamma_B \le \gamma_A\!+\! 1$. \end{Theorem}

\begin{proof} Since the extension is split, every indecomposable $B$-module is a direct summand of an indecomposable $A$-module.  According to \cite[(4.2)]{dP2}, the algebra $B$ is therefore tame. Thus, for each $d>0$, there are parametrizing families
\[ f_i : \A^1 \lra \modd_B^d \ \ \ \ ; \ \ \ \ 1 \le i \le \mu_B(d).\]
We set $V_i :=   \overline{{\rm im} \, f_i}$, and apply Proposition \ref{GD4} and Lemma \ref{GD3}(1) to obtain
\[ (\dagger) \ \ \ \ \ \ \ \ |{\rm im} f_i \cap \GL_d(k)\dact x| < \infty \ \ \ \ \ \forall \ x \in \modd^d_B.\]
Now let $X_1, \ldots , X_{\ell(d)}$, $F_d$ be the bimodules and the finite subset of $\ind_B^d$ provided by Lemma \ref{SpE1}, respectively. For $j \in \{1, \ldots, \ell(d)\}$ let
\[ g_j : \A^1 \lra \modd_A^{r_j} \ \ ; \ \ \ r_j := \rk_{k[T]}(X_j)\]
be the corresponding parametrizing morphisms. We put $W_j := \overline{{\rm im} \, g_j}$ and choose $\ell(d)$ to be minimal subject to properties (a) and (b) of Lemma \ref{SpE1}.

Chevalley's Theorem provides dense open subsets $O_i \subseteq V_i$ and $U_j \subseteq W_j$ with $O_i \subseteq \im \, f_i$ and $U_j \subseteq \im\,  g_j$.

Given $(i,j) \in \{1,\ldots,\mu_B(d)\}\times\{1,\ldots,\ell(d)\}$, we let $\cS(i,j) \subseteq \im f_i\times \im g_j$ be the set of all those pairs $(M,N) \in \modd^d_B\times\modd^{r_j}_A$ such that $M$ is isomorphic to a direct summand of $N|_B$ and $N$ is isomorphic to a direct summand of $A\!\otimes_B\!M$. In view of Lemma \ref{GD5} and the arguments of \cite[p.183]{dP1}, this set is constructible. We denote by $\pi_{(i,j)} : \overline{\cS(i,j)} \lra V_i$ the restriction of the projection $V_i\times W_j\lra V_i$.

\medskip
(a) {\it We have $\dim \overline{\cS(i,j)} \le 1$}.

\smallskip
\noindent
Since $\dim V_i = 1 = \dim W_j$, we have $\dim \overline{\cS(i,j)} \le 2$. Suppose equality to hold. Then $ \overline{\cS(i,j)} = V_i\times W_j$ is an irreducible variety.  Thus, $O_i\times U_j \subseteq \im \, f_i \times \im \, g_j$ is an open, and thus dense, subset of $\overline{\cS(i,j)}$. As $\cS(i,j)$ is constructible, it also contains a
dense, open subset of the irreducible variety $\overline{\cS(i,j)}$. Its non-empty intersection with $O_i\times U_j$ will be denoted $O$.

Let $(M,N) \in O$ and consider an element $(V,W) \in O\cap \pi_{(i,j)}^{-1}(\pi_{(i,j)}(M,N))$. Then $M = V$, $M$ is a direct summand of $W|_B$, and $W$ is direct summand of the $A$-module $A\!\otimes_B\!M$. The Theorem of Krull-Remak-Schmidt provides a finite subset $\cF \subseteq \modd^{r_j}_A$ such that
\[ O\cap \pi_{(i,j)}^{-1}(\pi_{(i,j)}(M,N)) \subseteq \{M\} \times (U_j\cap \GL_{r_j}(k)\dact\cF).\]
Thanks to Lemma \ref{SpE1}(3), the set $U_j\cap \GL_{r_j}(k)\dact \cF$ is finite, so that $O\cap \pi_{(i,j)}^{-1}(\pi_{(i,j)}(M,N))$ also has this property.

Let $Z \subseteq \pi_{(i,j)}^{-1}(\pi_{(i,j)}(M,N))$ be an irreducible component containing $(M,N)$. Then $(M,N) \in O\cap Z$, so that $O\cap Z$ is a dense, open subset of $Z$. Since $O\cap Z$ is finite, we obtain
$\dim Z = \dim Z\cap O = 0$, implying
\[ \dim_{(M,N)} \pi_{(i,j)}^{-1}(\pi_{(i,j)}(M,N)) = 0 \ \ \ \ \ \ \ \forall \ (M,N) \in O.\]
Hence the generic fiber of the morphism $\pi_{(i,j)}$ is zero-dimensional, so that $\dim V_i \ge 2$, a contradiction cf.\ \cite[(I.\S8)]{Mu}.  \ \ \ \ $\diamond$

\medskip
(b)  {\it  We have  $V_i = \bigcup_{j=1}^{\ell(d)} \overline{\pi_{(i,j)}(\cS(i,j))}$}.

\smallskip
\noindent
Thanks to ($\dagger$), the set $O_i\cap \GL_d(k)\dact F_d$ is finite. Thus, $O_i \!\smallsetminus\! (O_i\cap \GL_d(k)\dact F_d)$ is a dense open subset of the one-dimensional irreducible variety $V_i$, and Lemma \ref{GD3}(2) implies that $O'_i  := (O_i\!\smallsetminus\!(O_i\cap \GL_d(k).F_d))\cap (V_i \cap \ind_A^d)$ enjoys the same property. Owing to Lemma \ref{SpE1}(1), we have
\[ O'_i \subseteq  \bigcup_{j=1}^{\ell(d)}\pi_{(i,j)}(\cS(i,j)),\]
whence
\[ V_i = \overline{O'_i} \subseteq  \bigcup_{j=1}^{\ell(d)} \overline{\pi_{(i,j)}(\cS(i,j))} \subseteq V_i,\]
as desired. \ \ \ \ $\diamond$

\medskip
\noindent
Let $f : X \lra Y$ be a continuous map between topological spaces, $Q\subseteq X$ be a subset. Since $f(\overline{Q}) \subseteq \overline{f(Q)}$, we have
\[  \overline{f(\overline{Q})} = \overline{f(Q)}.\]
Accordingly, a morphism $\gamma : \overline{\cS(i,j)} \lra V_i$ is dominant if and only if $\overline{\gamma(\cS(i,j))} = V_i$.

\medskip
(c) {\it Let $\sigma_{(i,j)} : \overline{\cS(i,j)} \lra W_j$ be induced by the projection $V_i\times W_j \lra W_j$. If $Z \subseteq \overline{\cS(i,j)}$ is an irreducible component such that $\pi_{(i,j)} :  Z \lra V_i$ is a dominant morphism, so is $\sigma_{(i,j)} : Z \lra W_j$}.

\smallskip
\noindent
Suppose the contrary, so that $\sigma_{(i,j)} : Z \lra W_j$ has a finite image. Thanks to (a) and the choice of $Z$, we have $\dim Z = 1$. Setting $r := \min \{ \dim_z\sigma_{(i,j)}^{-1}(\sigma_{(i,j)}(z)) \ ; \ z \in Z\}$, we conclude from upper semicontinuity of fiber dimension \cite[(I.\S8)]{Mu} that
\[ O := \{ z \in Z \ ; \ \dim_z \sigma_{(i,j)}^{-1}(\sigma_{(i,j)}(z)) = r\}\]
is a dense, open subset of the one-dimensional variety $Z$. On the other hand, the generic fiber dimension theorem implies  $\dim_z \sigma_{(i,j)}^{-1}(\sigma_{(i,j)}(z)) = 
\dim Z\!-\! \dim \overline{\sigma_{(i,j)}(Z)} = 1$ on a non-empty, open subset of $Z$, so that $r=1$.

Since the morphism $\pi_{(i,j)} : Z \lra V_i$ is dominant, its image intersects the open subset $O_i \subseteq \im f_i \subseteq V_i$. Thus, $\pi_{(i,j)}^{-1}(O_i)$ is a dense, open subset of $Z$. The constructible set $\cS(i,j)$ contains a dense, open subset $U$ of $\overline{\cS(i,j)}$ (cf.\ \cite[(AG.1.3)]{Bo}). The assumption $U\cap Z = \emptyset$ implies that $U$ is contained in the union of the irreducible components of $\overline{\cS(i,j)}$ different from $Z$. Thus, $\overline{U} \subsetneq \overline{\cS(i,j)}$, a contradiction.
As a result,
\[ \tilde{O} := \pi_{(i,j)}^{-1}(O_i)\cap U \cap O\]
is a dense, open subset of $Z$.

Let $x := (M,N)$ be an element of $\tilde{O}$. If $(P,Q) \in \tilde{O}\cap \sigma_{(i,j)}^{-1}(\sigma_{(i,j)}(M,N))$, then we have $Q=N$ and $P$ is a direct summand of $N|_B$. Consequently, the Theorem of Krull-Remak-Schmidt provides a finite subset $\cF \subseteq \modd_B^d$ such that
\[ \tilde{O}\cap \sigma_{(i,j)}^{-1}(\sigma_{(i,j)}(M,N)) \subseteq (O_i\cap {\rm GL}_d(k)\dact\cF)\times \{N\}.\]
Thanks to property ($\dagger$) the right-hand set is finite. Since $x \in \tilde{O}$, the set $\tilde{O}\cap \sigma_{(i,j)}^{-1}(\sigma_{(i,j)}(x))$ intersects every irreducible component $X$ of
$\sigma_{(i,j)}^{-1}(\sigma_{(i,j)}(x))$ containing $x$ non-trivially. Let $\cA$ be the set of these components. Consequently,
\[ \dim_x\sigma_{(i,j)}^{-1}(\sigma_{(i,j)}(x)) = \max_{X \in \cA} \dim X = \max_{X \in A} \dim X\cap \tilde{O} = 0,\]
which contradicts the fact that $x \in O$. \ \ \ \ $\diamond$

\medskip

(d) {\it Let $\varphi : V \lra W$ be a dominant morphism of one-dimensional irreducible varieties. Then $\varphi$ is open}.

\smallskip
\noindent
Let $\emptyset \neq O \subseteq V$ be an open set. Then $\varphi(O)$ is a constructible subset of a one-dimensional irreducible variety and thus by Lemma \ref{GD2} finite or cofinite. In the former case, we have
\[ W = \overline{\varphi(V)} =   \overline{\varphi(\overline{O})} \subseteq \overline{\varphi(O)} = \varphi(O),\]
a contradiction. Thus, $\varphi(O)$ is cofinite and in particular open.  \ \ \ \ $\diamond$

\medskip
(e) {\it Given $j \in \{1, \ldots, \ell(d)\}$, the set
\[ \Gamma_j := \{ i \in \{1, \ldots, \mu_B(d)\} \ ; \  V_i =  \overline{\pi_{(i,j)}(\cS(i,j))}\}\]
has cardinality $|\Gamma_j| \le \dim_kA$}.

\smallskip
\noindent
Let $r \in \Gamma_j$. By assumption, there exists an  irreducible component $Z_r \subseteq \overline{\cS(r,j)}$ such that the morphism $\pi_{(r,j)} : Z_r \lra V_r$ is dominant. Thanks to (c), we also have a dominant morphism $\sigma_{(r,j)} : Z_r \lra W_j$.

In view of Lemma \ref{GD3}(2), $\ind_B^d\cap V_r$ is a dense, open subset of $V_r$, so that $\emptyset \ne \pi_{(r,j)}^{-1}( {\rm ind}_B^d\cap V_r) \subseteq Z_r$ enjoys the same property. By constructibility of $\cS(r,j)$, there exists a dense open subset $\tilde{U}_r \subseteq \cS(r,j)$ of $\overline{\cS(r,j)}$. As $Z_r$ is a component of $\overline{\cS(r,j)}$, this set intersects $Z_r$
non-trivially. Put $U_r := V_r \cap \ind_B^d\!\smallsetminus\! (\bigcup_{\ell \ne r} {\rm GL}_d(k)\dact V_\ell)$. In view of (d) and Lemma \ref{GD3}(3),
\[ O_r := \sigma_{(r,j)}(\tilde{U}_r \cap \pi_{(r,j)}^{-1}(U_r)) \subseteq \sigma_{(r,j)}(\cS(r,j))\]
is a dense, open subset of $W_j$. Thus, $\tilde{O} := \bigcap_{r \in \Gamma_j}O_r$ also has these properties.

We fix an element $s \in \Gamma_j$. Since $\sigma_{(s,j)} : Z_s \lra W_j$ is dominant, we have $\sigma_{(s,j)}^{-1}(\tilde{O}) \neq \emptyset$. As $\tilde{U}_s$ meets $Z_s$, it follows that $U'_s:=\tilde{U}_s\cap \sigma^{-1}_{(s,j)}(\tilde{O}) \subseteq \cS(s,j)$ is a dense, open subset of the irreducible variety $Z_s$. In view of (a) and (d), $O'_s:= \pi_{(s,j)}(U'_s)$ is a dense open subset of $V_s$.

Thanks to Lemma \ref{GD3}(2), $V_s \cap \ind_B^d$ is a dense open subset of $V_s$. Its intersection with $O'_s$ yields a dense open subset $O''_s \subseteq O'_s \cap {\rm ind}_B^d$ of $V_s$. In particular, we have $V_s = O''_s \cup \cF_s$ for some finite subset $\cF_s \subseteq V_s$.  

Let $M$ be an element of $O''_s$. Then $M$ is indecomposable and $M = \pi_{(s,j)}(M,N)$ with $(M,N) \in U'_s$. Thus, $(M,N) \in \cS(s,j)$ and $N = \sigma_{(s,j)}(M,N) \in O_r \subseteq \sigma_{(r,j)}(\cS(r,j))$ for every $r \in \Gamma_j$. Hence there exists, for every $r \in \Gamma_j$, a $B$-module $M_r \in  U_r$ such that $(M_r,N) \in \cS(r,j)$.

The $B$-modules $(M_r)_{r \in \Gamma_j}$ are indecomposable direct summands of $N|_B$. As each $M_r$ belongs to $U_r$, the $M_r$ are pairwise non-isomorphic. By the
Theorem of Krull-Remak-Schmidt, the module $\bigoplus_{r \in \Gamma_j} M_r$ is also a direct summand of $N|_B$. Since $\dim_k N \le \dim_k A\!\otimes_B\!M_s \le (\dim_kA)d$, we obtain
\[ d |\Gamma_j| \le \dim_k N \le d (\dim_kA),\]
so that $|\Gamma_j| \le \dim_kA$.  \ \ \ \ $\diamond$

\medskip
\noindent
Let $d>0$. Since each $V_i$ is irreducible, claim (b) provides a map $\omega : \{1, \ldots, \mu_B(d)\} \lra \{1, \ldots, \ell(d)\}$ such that $V_i = \overline{\pi_{(i,\omega(i))}(\cS(i,\omega(i))}$. We set $\Omega := \im\,\omega$. There exists a finite set $F'_d  \subseteq \modd_B^d$ such that
\[ \ind_B^d \subseteq \bigcup_{t \in \Omega}\bigcup_{i \in \omega^{-1}(t)}\GL_d(k)\dact V_i \cup \GL_d(k)\dact F'_d \subseteq
\bigcup_{t \in \Omega}\bigcup_{i \in \Gamma_t}\GL_d(k)\dact V_i \cup \GL_d(k)\dact F'_d.\]
Thanks to (e) and Lemma \ref{SpE1}, there exists $c>0$ such that the number $\mu_B(d)$ of one-parameter families is bounded by
\[ | \Omega | ( \max_{j \in \Omega} | \Gamma_j |)  \le \ell(d) (\dim_kA) \le c(\dim_kA)d^{\gamma_A}.\]
Consequently, the algebra $B$ has polynomial growth of rate $\gamma_B \le \gamma_A\!+\!1$. \end{proof}

\bigskip
\noindent
The values of $\gamma_A$ and $\gamma_B$ may be arbitrarily far apart. Let $A$ be a tame algebra. Then $A\!:\!k$ is a split extension with $\gamma_k=0$. 

The upper bound for $\gamma_B$ depends on the knowledge of the indecomposable constituents of $A\!\otimes_B\!V$ containing $V$ as a direct summand. The following result shows
how a better understanding of induced modules leads to improved estimates.

\bigskip

\begin{Corollary} \label{SpE4} Suppose that $A\!:\!B$ is an extension of $k$-algebras such that $A = \bigoplus_{i=1}^r A_i$ is a direct sum of $(B,B)$-modules with
$A_i \cong B^{(\id_B\otimes \zeta_i)}$ for some $\zeta_i \in {\rm Aut}(B)$. If $A$ is of polynomial growth, then $B$ is of polynomial growth with $\gamma_B \le \gamma_A$. \end{Corollary}

\begin{proof} Using Corollary \ref{SpE2} in the last paragraph of the above proof, we arrive at the stronger estimate. \end{proof}

\bigskip

\begin{Examples}  (1) The classical example of a split extension is the trivial extension $A := B\oplus M$ of $B$ by a $(B,B)$-bimodule $M$. In that case, $B$ is also a factor algebra
of $A$ and our prefatory remarks show that $B$ inherits polynomial growth from $A$ along with $\gamma_B \le \gamma_A$.

(2) \ Let $G$ be a finite group, $H \subseteq G$ be a subgroup, $S \subseteq G$ be a set of right coset representatives of $H$ in $G$ containing $1$. Then
\[ kG = kH \oplus (\bigoplus_{s \ne 1} kHs)\] 
is a direct sum of $(kH,kH)$-bimodules, so that the extension $kG\!:\!kH$ is split.

If $H \unlhd G$ is normal in $G$, then
\[ kG = \bigoplus_{s \in S} kHs\]
is a decomposition of $kH$-bimodules, with each summand being isomorphic to $kH^{(\id_{kH}\otimes s\dact^{-1})}$, where $s\dact^{-1}$ denotes the conjugation by $s^{-1}$. Suppose that $kG$ has polynomial growth. It now follows from Corollary \ref{SpE4} that $kH$ has polynomial growth with growth rate $\gamma_{kH} \le \gamma_{kG}$.

(3) \ Suppose that $\Char(k)=p>0$. The following example shows that extensions defined by ``group algebras'' of infinitesimal group schemes tend to behave differently. Let $\fb \subseteq \fsl(2)$ be the Borel subalgebra of upper triangular $(2\!\times\!2)$-matrices of the restricted Lie algebra $\fsl(2)$. The extension $U_0(\fsl(2))\!:\!U_0(\fb)$, given by the associated restricted
enveloping algebras, is not split: Let $\lambda \in \fb^\ast$ be the Steinberg weight, so that $U_0(\fsl(2))\!\otimes_{U_0(\fb)}\!k_\lambda$ is the Steinberg module. This module is known to be projective, so that its restriction $(U_0(\fsl(2))\!\otimes_{U_0(\fb)}\!k_\lambda)|_{U_0(\fb)}$ is also projective. If $U_0(\fsl(2))\!:\!U_0(\fb)$ was split, then $k_\lambda$ would be a direct summand of $(U_0(\fsl(2))\!\otimes_{U_0(\fb)}\!k_\lambda)|_{U_0(\fb)}$ and hence also projective. Since $\fb$ is not a torus, this is impossible. Note that $\gamma_{U_0(\fsl(2))}=1$, while $\gamma_{U_0(\fb)}=0$. \end{Examples}

\bigskip

\section{Separable Extensions}
In this section we show that polynomial growth is preserved under ascent via separable extensions.

\bigskip

\begin{Lemma} \label{SE1} Let $A\!:\!B$ be a separable extension, and suppose that $B$ is tame of polynomial growth. Then there exists a function $\ell : \NN \lra \NN_0$ with the following properties:
\begin{enumerate}
\item For every $d>0$ there exist a finite subset $F_d \subseteq \ind^d_A$ and $(B,k[T])$-bimodules $Y_1,\ldots,Y_{\ell(d)}$ that are free $k[T]$-modules
of rank $\le d$ such that for every $M \in \ind^d_A\!\smallsetminus\! \GL_d(k)\dact F_d$ there exist $i \in \{1,\ldots, \ell(d)\}$ and $\lambda : k[T] \lra k$ with
\begin{enumerate}
\item[(a)] $M$ is isomorphic to a direct summand of $(A\!\otimes_B\!Y_i)\!\otimes_{k[T]}\!k_\lambda$, and
\item[(b)] $Y_i\!\otimes_{k[T]}\!k_\lambda$ is isomorphic to a direct summand of $M|_B$. \end{enumerate}
\item The function $\ell$ has polynomial growth of rate $\gamma_\ell \le \gamma_B\!+\!1$. 
\item If $\ell(d)$ is minimal subject to (a) and (b), then 
\[ |\{Y_i\!\otimes_{k[T]}\!k_\lambda \ ; \lambda \in k\} \cap \GL_{r_i}(k)\dact x| < \infty\]
for all $i \in \{1,\ldots,\ell(d)\}$, $x \in \modd^{r_i}_B,\ r_i := \rk_{k[T]}(Y_i)$.\end{enumerate} \end{Lemma}

\begin{proof} (1) Given $j>0$, we let $X_1^j,\ldots, X^j_{\mu_B(j)}$ be the parametrizing $(B,k[T])$-modules and $G_j \subseteq \ind^j_B$ be the finite set of exceptional modules.

Let $d>0$, $M \in \ind^d_A$. Since the extension $A\!:\!B$ is separable, we have $M\!\mid\!A\!\otimes_B\!M$, and there exists an indecomposable direct summand $N$ of $M|_B$ such that $M\!\mid\!A\!\otimes_B\!N$. Hence $N \in \ind^j_B$ for some $j \in \{1,\ldots, d\}$. 

We consider the $(B,k[T])$-bimodules $\{X^j_i \  ;  \ 1\le j \le d, \  1\le i \le \mu_B(j)\}$. Let $\ell(d)$ be the number of these modules, which we denote $Y_1,\dots, Y_{\ell(d)}$. By definition, each $Y_i$ is a free $k[T]$-module of rank $\rk_{k[T]}(Y_i) \le d$. 

Consider the finite set
\[ U_d := \{ A\!\otimes_B\! N \ ; \ N \in \bigcup_{j=1}^d G_j\}.\]
For each $X \in U_d$ we pick a decomposition into indecomposable constituents. Let $\cU_d$ be the finite set of these constituents and put $F_d := \ind^d_A\cap \cU_d$.

Let $M \in \ind^d_A\!\smallsetminus\! \GL_d(k)\dact F_d$. By the above, there exist $j \in \{1,\ldots, d\}$ and an indecomposable direct summand $N\!\mid\!(M|_B)$
with $\dim_kN=j$, such that $M$ is isomorphic to a direct summand of $A\!\otimes_B\!N$. The assumption $N \in \GL_j(k)\dact G_j$ implies $M \in \GL_d(k)\dact F_d$, a contradiction. It follows that $N \cong Y_i\!\otimes_{k[T]}\!k_\lambda$ for some $i \in \{ 1, \ldots, \ell(d)\}$. 

(2) Since $B$ has polynomial growth, there exists $c>0$ such that
\[ \ell(d) = \sum_{j=1}^d\mu_{B}(j) \le c \sum_{j=1}^dj^{\gamma_B-1} \le c d d^{\gamma_B-1} \le cd^{\gamma_B}.\]
As a result, we have $\gamma_\ell \le \gamma_B\!+\!1$.  

(3) Suppose there is $x_0 \in \modd^{r_i}_B$ for some $i \in \{1,\ldots, \ell(d)\}$ such that the set $\cV_i:=\{Y_i\!\otimes_{k[T]}\!k_\lambda \ ; \lambda \in k\} \cap \GL_{r_i}(k)\dact x_0$ is infinite. Lemma \ref{GD2} then shows that $\cV_i$ is cofinite in $\{Y_i\!\otimes_{k[T]}\!k_\lambda \ ; \lambda \in k\}$. We write 
\[ \{Y_i\!\otimes_{k[T]}\!k_\lambda \ ; \lambda \in k\}  = \cV_i \cup G_{r_i}\]
for some finite set $G_{r_i} \subseteq \modd^{r_i}_B$. For every $N \in G_{r_i}\cup\{x_0\}$, we fix a decomposition of $A\!\otimes_B\!N$ into indecomposables and consider the finite subset $\cF_d\subseteq \ind^d_A$ of $d$-dimensional indecomposable constituents.

Let $M$ be an element of $\ind^d_A\!\smallsetminus\! \GL_d(k)\dact (F_d\cup\cF_d)$. Since $M\not \in \GL_d(k)\dact F_d$, $M$ is a direct summand of some $(A\!\otimes_B\!Y_j)\!\otimes_{k[T]}\!k_\lambda$, while $M \not \in \GL_d(k)\dact \cF_d$ ensures that $j\ne i$. This, however, contradicts the minimality of $\ell(d)$. \end{proof}

\bigskip
\noindent
Given a finite group $G$ that acts on a $k$-algebra $B$ via automorphisms, we denote by $B\!\ast\! G$ the {\it skew group algebra}, cf.\ \cite[(III.4)]{ARS}. For separable extensions of the form $B\!\ast\!G\!:\!B$, the estimate for the growth of the function $\ell$ may be strengthened. 

In the following, we denote by $\lfloor q\rfloor$ the least integer greater than or equal to $q \in \QQ$. 

\bigskip

\begin{Lemma} \label{SE2} Let $G$ be a finite group acting on a $k$-algebra $B$ via automorphismsm. Suppose that $B$ has polynomial growth, and that $p\nmid {\rm ord}(G)$.
Then there exists a function $\ell : \NN \lra \NN_0$ with the following properties:
\begin{enumerate}
\item For every $d>0$ there exist a finite subset $F_d \subseteq \ind^d_{B\ast G}$ and $(B,k[T])$-bimodules $Y_1,\ldots,Y_{\ell(d)}$ that are free $k[T]$-modules
of rank $\le d$ such that for every $M \in \ind^d_{B\ast G}\!\smallsetminus\! \GL_d(k)\dact F_d$ there exist $i \in \{1,\ldots, \ell(d)\}$ and $\lambda : k[T] \lra k$ such that
\begin{enumerate}
\item[(a)] $M$ is isomorphic to a direct summand of $(B\!\ast\!G\!\otimes_B\!Y_i)\!\otimes_{k[T]}\!k_\lambda$, and
\item[(b)] $Y_i\!\otimes_{k[T]}\!k_\lambda$ is isomorphic to a direct summand of $M|_B$. \end{enumerate}
\item The function $\ell$ has polynomial growth of rate $\gamma_\ell \le \gamma_B$. 
\item If $\ell(d)$ is minimal subject to (a) and (b), then 
\[ |\{Y_i\!\otimes_{k[T]}\!k_\lambda \ ; \lambda \in k\} \cap \GL_{r_i}(k)\dact x| < \infty\]
for all $i \in \{1,\ldots,\ell(d)\}$, $x \in \modd^{r_i}_B,\ r_i := \rk_{k[T]}(Y_i)$.\end{enumerate} \end{Lemma}

\begin{proof} We put $A := B\!\ast\! G$ for ease of notation. Since $p$ does not divide ${\rm ord}(G)$, the extension $A\!:\!B$ is separable (cf.\ Proposition \ref{SP1} below).

Given $j>0$, we let $X_1^j,\ldots, X^j_{\mu_B(j)}$ be parametrizing $(B,k[T])$-bimodules, $G_j \subseteq \ind^j_B$ be the finite set of exceptional modules.

Let $d>0$, $M \in \ind^d_A$. Then there exists an indecomposable direct summand $N$ of $M|_B$ such that $M$ is a direct summand of $A\!\otimes_B\!N$.
Since $(A\!\otimes_B\!N)|_B \cong \bigoplus_{g \in G}N^{(g)}$, there exists a subset $S_M \subseteq G$ with $M|_B \cong  \bigoplus_{g \in S_M}N^{(g)}$.
As a result $\dim_kM = |S_M| \dim_k N$, and we conclude that $N \in \ind^j_B$, where $d=rj$ and $r\le {\rm ord}(G)=:|G|$. We consider the $(A,k[T])$-bimodules 
\[ A\!\otimes_B\!X^{\lfloor d/r\rfloor}_i \ \ \ ; \ \ \ 1\le r \le |G| \ , \ 1\le i \le \mu_B(\lfloor d/r\rfloor). \]
Let $\ell(d)$ be the number of these modules. Since $B$ has polynomial growth, there exists $c>0$ such that
\[ \ell(d) = \sum_{r=1}^{|G|}\mu_B(\lfloor d/r\rfloor) \le c \sum_{r=1}^{|G|}\lfloor d/r\rfloor^{\gamma_B-1} \le c |G| d^{\gamma_B-1}.\]
As a result, we have $\gamma_\ell \le \gamma_B$.  The remaining assertions follow as in Lemma \ref{SE1}.\end{proof}

\bigskip

\begin{Theorem} \label{SE3} Let $A\!:\!B$ be a separable extension. If $B$ is of polynomial growth, then $A$ has polynomial growth with 
$\gamma_A \le \gamma_B\!+\!1$. \end{Theorem}

\begin{proof} Since the extension $A\!:\!B$ is separable, every indecomposable $A$-module is a direct summand of  $A\!\otimes_B\!N$ for some indecomposable
$B$-module $N$. It thus follows from \cite[(4.2)]{dP2} that $A$ is tame. The arguments of Theorem \ref{SpE3} now apply mutatis mutandis.
We indicate the requisite changes.

Let $d>0$. There are parametrizing morphisms
\[ f_i : \A^1 \lra \modd_A^d \ \ \ \ ; \ \ \ \ 1 \le i \le \mu_A(d).\]
We set $V_i :=   \overline{{\rm im} \, f_i}$, and apply Proposition \ref{GD4} and Lemma \ref{GD3}(1) to obtain
\[ (\dagger) \ \ \ \ \ \ \ \ |{\rm im} f_i \cap \GL_d(k)\dact x| < \infty \ \ \ \ \ \forall \ x \in \modd^d_B.\]
Now let $Y_1, \ldots , Y_{\ell(d)}$, $F_d$ be the $(B,k[T])$-bimodules and the finite subset of $\ind_B^d$ provided by Lemma \ref{SE1}, respectively. For $j \in \{1, \ldots, \ell(d)\}$ let
\[ g_j : \A^1 \lra \modd_B^{r_j} \ \ ; \ \ \ r_j := \rk_{k[T]}(X_j)\le d\]
be the corresponding parametrizing morphisms. We put $W_j := \overline{{\rm im} \, g_j}$ and choose $\ell(d)$ to be minimal subject to properties (a) and (b) of Lemma \ref{SE1}.

Chevalley's Theorem provides dense open subsets $O_i \subseteq V_i$ and $U_j \subseteq W_j$ with $O_i \subseteq \im \, f_i$ and $U_j \subseteq \im\,  g_j$.

Given $(i,j) \in \{1,\ldots,\mu_B(d)\}\times\{1,\ldots,\ell(d)\}$, we denote by $\cS(i,j) \subseteq V_i\times W_j$ the set of all pairs $(M,N) \in \im f_i\times\im g_j$ such that $M$ is isomorphic to a direct summand of $A\otimes_B\!N$ and $N$ is isomorphic to a direct summand of $M|_B$. In view of Lemma \ref{GD5} and the arguments of \cite[p.183]{dP1}, this set is constructible. We denote by $\pi_{(i,j)} : \overline{\cS(i,j)} \lra V_i$ the restriction of the projection $V_i\times W_j\lra V_i$.

Assertions (a)-(d) of Theorem \ref{SpE3} can now be shown to hold in this context as well. In (e), we reach the conclusion that $\bigoplus_{r\in \Gamma_j}M_r$
is a direct summand of $A\!\otimes_B\!N$. Since $r_j \le d$ this implies
\[ d |\Gamma_j| \le \dim_k A\!\otimes_B\!N \le (\dim_kA)r_j \le (\dim_kA)d.\]
The remaining arguments of Theorem \ref{SpE3} may be adopted verbatim. \end{proof}

\bigskip

\section{Applications}
\subsection{Smash products} Our first application concerns the smash product of Hopf algebras. We refer the reader to \cite{Mo} for general facts concerning Hopf algebras. In the following,
we shall write $\eta$ and $\varepsilon$ for the antipode and co-unit of a Hopf algebra and use Sweedler notation $\Delta(h) = \sum_{(h)}h_{(1)}\!\otimes\!h_{(2)}$ for calculations involving
the comultiplication.

Let $H$ be a finite-dimensional Hopf algebra, $B$ be an $H$-module algebra. By definition, $B$ is an $H$-module with action $(h,b) \mapsto h\dact b$ such that
\begin{enumerate}
\item[(a)] $h\dact (ab) = \sum_{(h)}(h_{(1)}\dact a) (h_{(2)}\dact b)$ for all $h \in H, \, a,b \in B$, and
\item[(b)] $h\dact 1 = \varepsilon(h)1$ for all $h \in H$. \end{enumerate}
We consider the smash product $B \sharp H$, with underlying $k$-space $B\!\otimes_k\!H$ and multiplication
\[ (a\!\otimes\! h).(b\!\otimes\! h') = \sum_{(h)} a (h_{(1)}\dact b)\!\otimes\! h_{(2)}h' \ \ \ \ \forall \ h,h' \in H, \, a,b \in B.\]
Recall that the canonical maps $b \mapsto b\!\otimes\!1$ and $h \mapsto 1\!\otimes\! h$ are injective homomorphisms of algebras \cite[Chap.4]{Mo}. In this section we study the extension $B\sharp H\!:\!B$.

A Hopf algebra $H$ is referred to as {\it cosemisimple} if its dual algebra $H^\ast$ is semisimple. Recall that an extension $A\!:\!B$ of $k$-algebras is a {\it free Frobenius extension
of first kind}, provided 
\begin{enumerate}
\item[(a)] $A$ is a finitely generated free $B$-module, and
\item[(b)]  there exists an isomorphism $\varphi : A \lra \Hom_B(A,B)$ of $(A,B)$-bimodules. \end{enumerate}
In this case, the $B$-module $A$ is known to afford dual bases $\{x_1,\ldots,x_n\}$ and $\{y_1,\ldots,y_n\}$ such that $\varphi(x_i)(y_j) = \delta_{i,j}$. The map
\[ c_{A:B} : A \lra A \ \ ; \ \  b \mapsto \sum_{i=1}^n y_ibx_i,\]
the so-called {\it Gasch\"utz-Ikeda operator}, does not depend on the choice of the bases.

\bigskip

\begin{Prop} \label{SP1} Let $B$ be an $H$-module algebra. Then the following statements hold:
\begin{enumerate}
\item $B\sharp H\!:\!B$ is a free Frobenius extension of first kind.
\item If $H$ is cosemisimple, then the extension $B\sharp H\!:\!B$ is split. 
\item If $H$ is semisimple, then the extension $B\sharp H\!:\!B$ is separable.  \end{enumerate} \end{Prop}

\begin{proof} (1) Let $\lambda \in H^*$ be a non-zero left integral of $H^*$, that is, a linear form satisfying
\[ (\ast) \ \ \ \ \ \ \ f\ast\lambda = f(1)\, \lambda \ \ \ \forall \ f \in H^* ,\]
where $f\!\ast\!g(h)=\sum_{(h)}f(h_{(1)})g(h_{(2)})$ denotes the convolution of $f,g \in H^\ast$. We define a $k$-linear map $\pi : B\sharp H \lra B$ by means of
\[ \pi(a \otimes h) := a\, \lambda (h)  \ \ \ \ \ \ \forall \ a \in B,\ h \in H. \]
Note that $\pi((a \otimes 1).(b \otimes h)) = \pi(ab \otimes h) = ab \lambda(h) = a \pi(b \otimes h)$ as well as
\begin{eqnarray*}
\pi((a \otimes h).(b \otimes 1)) & = & \sum_{(h)} \pi(a(h_{(1)}\dact b)
\otimes h_{(2)} 1) = \sum_{(h)} a (h_{(1)} \dact b) \lambda(h_{(2)}) \\
& = & a\, ((\sum_{(h)} h_{(1)} \lambda(h_{(2)}))\dact b).
\end{eqnarray*}
Given $f \in H^*$ we obtain, observing ($\ast$),
\[ f( \lambda (h)1) = f(1) \, \lambda(h) = \sum_{(h)} f(h_{(1)}) \lambda(h_{(2)}) = f(\sum_{(h)} h_{(1)} \lambda(h_{(2)})), \]
so that the last term above equals
\[ a \, \lambda(h) b = \pi(a \otimes h)b. \]
Since $\lambda$ is a Frobenius homomorphism of the Frobenius algebra $H$ (cf.\ \cite{LS}), there exist bases $\{x_1, \ldots , x_n \}$;  $\{y_1 , \ldots , y_n \}$ of $H$ such that $\lambda(x_i \, y_j) = \delta_{i,j}$ for $1 \leq i,j \leq n$. Observing
\begin{eqnarray*}
\pi((1 \otimes x).(1 \otimes y)) & = & \pi (\sum_{(x)} \varepsilon(x_{(1)})1_B \otimes x_{(2)}y) = \sum_{(x)}\varepsilon (x_{(1)})
\lambda (x_{(2)}y) \\
& = & \lambda(\sum_{(x)}\varepsilon (x_{(1)})x_{(2)}y) = \lambda(xy),
\end{eqnarray*}
we obtain $\pi((1 \otimes x_i).(1\otimes y_j)) = \delta_{i,j} \ \ \ 1 \leq i,j \leq n. $ By definition, the set $\{1\otimes x_1, \ldots , 1 \otimes x_n \}$ is a basis for the left $B$-module $B\sharp H$. Suppose that
$0 = \sum_{i=1}^n (1\otimes y_i).(b_i\otimes 1)$ for some $b_i \in B$. Then we have $0 = \pi((1\otimes x_j).0) = \sum_{i=1}^n \pi((1\otimes x_j)(1\otimes y_i))b_i = b_j$ for $1\le j \le n$. As a result, the
$k$-vector space $\sum_{i=1}^n(1\otimes y_i).B$ has dimension $\dim_kB\sharp H$, so that $ \{1 \otimes y_1, \ldots , 1 \otimes y_n \} $ generates $B \sharp H$ as a right $B$-module. The result now follows from \cite[(1.2)]{BF}.

(2) By assumption, the algebra $H^\ast$ is semisimple. Recall that the left $H$-comodule $H$ obtains the structure of a right $H^*$-module via
\[ h.\psi := \sum_{(h)} \psi(h_{(1)})h_{(2)} \ \ \ \ \ \ \forall \ h \in H,\ \psi \in H^\ast.\]
Consequently, $k1$ is an $H^\ast$-submodule of $H$ and thus possesses a direct complement $V$. By general theory \cite[(1.6.4)]{Mo}, $V$ is a left subcomodule of $H$, so that
$\Delta(V) \subseteq H\!\otimes_k\!V$. There results a vector space decomposition
\[ B\sharp H = B\oplus (B\!\otimes_k\!V),\]
which obviously is a decomposition of left $B$-modules. It remains to show that $B\!\otimes_k\!V$ is a right $B$-submodule of $B\sharp H$.
Given $a,b \in B$ and $v \in V$, we have
\[ (a\otimes v).(b\otimes1) = \sum_{(v)} a(v_{(1)}\dact b)\otimes v_{(2)} \in B\!\otimes_k\!V,\]
as desired.  

(3) We denote by $c_H$ and $c_{B\sharp H:B}$ the Gasch\"utz-Ikeda operators of the Frobenius extensions $H\!:\!k$ and $B\sharp H\!:\!B$, respectively. The arguments of (1) imply
\[ c_{B\sharp H:B}(1\!\otimes\!h) = \sum_{i=1}^n (1\!\otimes\!y_i).(1\!\otimes h).(1\!\otimes\!x_i) = 1\!\otimes c_H(h)\]
for every $h \in H$. Since $H$ is semisimple, the Frobenius extension $H\!:\!k$ is separable, and there exists an element $z$ of the center of $H$ such that $1= c_H(z) = \sum_{i=1}^n y_izx_i$, cf.\ \cite[(2.6)]{Fa1}. Consequently, 
\[ 1\!\otimes\! 1 = c_{B\sharp H:B}(1\!\otimes\!z) = (1\!\otimes\!z).c_{B\sharp H:B}(1\!\otimes\!1) = c_{B\sharp H:B}(1\!\otimes\!1).(1\!\otimes\!z).\]
As $c_{B\sharp H:B}(1\!\otimes\!1)$ belongs to the centalizer $C_{B\sharp H:B}(B)$, the element $1\!\otimes\!z$ enjoys the same property, and \cite[(2.6)]{Fa1} (cf.\ also \cite[(2.18)]{HS})
shows that the extension is separable. \end{proof}

\bigskip

\begin{Prop} \label{SP2} Let $B$ be an $H$-module algebra. Then the following statements hold:
\begin{enumerate}
\item If $H$ is cosemisimple and $B\sharp H$ is of polynomial growth, then $B$ is of polynomial growth with $\gamma_B \le \gamma_{B\sharp H}\!+\!1$.
\item If $H$ is semisimple and $B$ is of polynomial growth, then $B\sharp H$ is of polynomial growth with $\gamma_{B\sharp H} \le \gamma_B\!+\!1$. \end{enumerate} \end{Prop}

\begin{proof} (1) Thanks to Proposition \ref{SP1}(2), the extension $B\sharp H\!:\!B$ is split. Consequently, Theorem \ref{SpE3} implies our result. 

(2) In view of Proposition \ref{SP1}(3), the extension $B\sharp H\!:\!B$ is separable and the assertion follows from Theorem \ref{SE3}. \end{proof} 

\bigskip
\noindent
Our next result refines \cite[(3.3)]{dP1}.

\bigskip

\begin{Cor} \label{SP3} Let $G$ be a finite group acting on a $k$-algebra $B$. Then the following statements hold:
\begin{enumerate}
\item If $B\!\ast\!G$ has polynomial growth, then $B$ has polynomial growth with $\gamma_B \le \gamma_{B\ast G}$.
\item Suppose that $p\nmid {\rm ord}(G)$. Then $B$ has polynomial growth if and only if $B\!\ast\!G$ has polynomial growth with $\gamma_{B\ast G} = \gamma_B$. \end{enumerate}\end{Cor}

\begin{proof} (1) Since $B\!\ast \!G = \bigoplus_{g\in G}Bg$ is a $(B,B)$-direct decomposition of $B\!\ast\!G$ with $Bg \cong B^{(\id_B\otimes g^{-1}\dact)}$, the assertion is a direct consequence 
of Corollary \ref{SpE4}.

(2) By replacing the reference to Lemma \ref{SE1} by Lemma \ref{SE2} in the proof of Theorem \ref{SE3}, we conclude that $B$ being of polynomial growth implies that $B\!\ast\!G$ has polynomial growth with $\gamma_{B\ast G} \le \gamma_B$. Part (1) provides the reverse direction as well as the reverse inequality. \end{proof}  

\bigskip
\noindent
A finite group scheme $\cG$ is said to be {\it linearly reductive}, provided its algebra of measures $k\cG$ is semisimple. If $B$ is an algebra, viewed as a functor $R \mapsto B\!\otimes_k\!R$ from commutative $k$-algebras to $k$-algebras, then there is the notion of $\cG$ acting on $B$ via automorphisms. This is quivalent to $B$ being a $k\cG$-module algebra. For the general theory of affine group schemes, we refer to \cite{Ja,Wa}. 

\bigskip

\begin{Cor} \label{SP4} Let $\cG$ be a linearly reductive group scheme that operates on $B$ via automorphisms. If $B$ has polynomial growth, then $B\sharp k\cG$ has polynomial growth of rate $\gamma_{B\sharp k\cG} \le \gamma_B$. \end{Cor}

\begin{proof} By general theory \cite[(6.8)]{Wa}, the group scheme $\cG$ is a semi-direct product
\[ \cG = \cG^0\rtimes \cG_{\rm red},\]
with an infinitesimal, normal subgroup $\cG$ and a reduced group scheme $\cG_{\rm red}$. This implies that 
\[ k\cG = k\cG^0\!\ast\!\cG(k)\]
is a skew group algebra. Since $\cG$ is linearly reductive, Nagata's Theorem \cite[(IV.\S3.3.6)]{DG} ensures that $\cG^0$ is diagonalizable and $p\nmid{\rm
ord}(\cG(k))$. 

We first assume that $\cG$ is infinitesimal, so that $\cG=\ccD$ is diagonalizable. By van den Bergh's Theorem \cite{vdB} (see also \cite[p.167]{Mo}), we have
\[ (B\sharp k\ccD)\sharp k\ccD^\ast  \cong B\!\otimes_k\! \End_k(k\ccD),\]
so that the two-fold smash product is of polynomial growth with rate $\gamma_{(B\sharp k\ccD)\sharp k\ccD^\ast} = \gamma_B$. 
Since $\ccD$ is diagonalizable, $k\ccD^\ast = kX(\ccD)$ is the group algebra of the character group $X(\ccD)$ of $\ccD$. Thus,
$(B\sharp k\ccD)\sharp k\ccD^\ast = (B\sharp k\ccD)\!\ast\!X(\ccD)$, and Corollary \ref{SP3}(1) shows that $B\sharp k\ccD$ has
polynomial growth of rate $\gamma_{B\sharp k\ccD} \le \gamma_B$. 

In the general case, we write $\cG = \ccD \rtimes G$, where $\ccD$ is diagonalizable and $G$ is finite with $p\nmid {\rm ord}(G)$. By the first part, the algebra $C:= B\sharp k\ccD$
is of polynomial growth with rate $\gamma_C \le \gamma_B$. Since $B\sharp k\cG \cong C\!\ast\!G$, the assertion now follows from Corollary \ref{SP3}(2).  \end{proof}

\bigskip

\subsection{Blocks of group algebras}
In preparation for our discussion of finite group schemes, we briefly recall the relevant result for the ``classical'' case concerning finite groups.  Recall that polynomial growth and domesticity are preserved under Morita equivalence. 

Let $G$ be a finite group, $\cB \subseteq kG$ be a block of the group algebra $kG$. In the sequel, we denote by $D_\cB \subseteq G$ the defect group of $\cB$, a $p$-subgroup of $G$.
If $kG$ affords a tame block $\cB$, then $p=2$ and $D_\cB$ is dihedral, semidihedral, or generalized quaternion, cf.\ \cite[(4.4.4)]{Be}. We illustrate the use of Theorem \ref{SpE3} via the following well-known result, which underscores the scarcity of representation-infinite blocks of polynomial growth:

\bigskip

\begin{Thm} Suppose that $\Char(k)=2$. Let $\cB \subseteq kG$ be a representation-infinite block. Then the following statements are equivalent:
\begin{enumerate}
\item The block $\cB$ is of polynomial gowth.
\item $D_\cB \cong \ZZ/(2)\!\times\!\ZZ/(2)$.
\item $\cB$ is Morita equivalent to $k(\ZZ/(2)\!\times\!\ZZ/(2))$, $kA_4$, or to the principal block of $kA_5$.
\item The block $\cB$ is domestic. \end{enumerate}\end{Thm}

\begin{proof} (1) $\Rightarrow$ (2) We put $V_4 := \ZZ/(2)\!\times\!\ZZ/(2)$. Let $C_G(D_\cB)$ be the centralizer of $D_\cB$ in $G$ and consider the subgroup $H:= D_\cB C_G(D_\cB)$. 
In view of Brauer's First Main Theorem (cf.\ \cite[Thm.2,p.103]{Al} and its succeeding remark), there exists a block $\fb \subseteq \cB$ of $kH$ such that the extension $\cB\!:\!\fb$ is split. Theorem \ref{SpE3} now shows 
that $\fb$ is of polynomial growth. Thanks to \cite[(6.4.6)]{Be}, the algebras $\fb$ and $kD_{\cB}$ are Morita equivalent, so that $kD_\cB$ is of polynomial growth. Since $\cB$ is representation-infinite, the group $D_\cB$ is not cyclic, see \cite[(6.5)]{Be}. Hence \cite[Prop.1]{Sk1} implies $D_\cB \cong V_4$. 

(2) $\Rightarrow$ (3) Let $e$ be the inertial index of $\cB$. General theory implies $e \in \{1,3\}$. If $e=1$, then \cite[(6.6.1)]{Be} provides an isomorphism $\cB \cong \Mat_n(kV_4)$. If $e=3$, then \cite[(6.6.3)]{Be} shows that $\cB$ is Morita equivalent to the group algebra $kA_4$ of the alternating group $A_4$ on $4$ letters, or to the principal block of $kA_5$. 

(3) $\Rightarrow$ (4) Since $V_4$ is the Sylow-$2$-subgroup of $V_4,A_4$ and $A_5$, \cite[(4.4.4)]{Be} implies that $\cB$ is domestic. 

(4) $\Rightarrow$ (1) trivial. \end{proof}

\bigskip

\subsection{Group schemes of domestic representation type}
We shall use our results to determine (up to a linearly reductive normal subgroup) all finite $k$-group schemes $\cG$ of odd characteristic, whose algebra $k\cG$ of measures is representation-infinite and domestic. As in the case of finite groups, non-domestic group schemes do not have polynomial growth.
 
Let $(\Lambda,\varepsilon)$ be an augmented $k$-algebra. There is a unique block $\cB_0(\Lambda)$ such that $\varepsilon(\cB_0(\Lambda)) \ne (0)$. This block is referred to as the {\it principal block} of $\Lambda$. If the finite group $G$ acts on $\Lambda$
via automorphisms of augmented algebras, then $\Lambda\!\ast\!G$ is augmented with augmentation $\varepsilon_G$ satisfying
\[ \varepsilon_G(\lambda g) = \varepsilon(\lambda) \ \ \ \ \ \ \forall \ \lambda \in \Lambda, \ g \in G.\]

\bigskip

\begin{Lem} \label{GS2} There exists a $(\Lambda,\Lambda)$-bimodule $X$ such that $\cB_0(\Lambda\!\ast\!G) = \cB_0(\Lambda)\oplus X$. \end{Lem}

\begin{proof} This follows directly from the proof of \cite[(5.1.2)]{Fa2}. \end{proof}

\bigskip
\noindent
Let $\Lambda$ be a $k$-algebra. The split extension $T(\Lambda):= \Lambda\ltimes\Lambda^\ast$ of $\Lambda$ by its dual bimodule $\Lambda^\ast$ is referred to as the {\it trivial extension} of $\Lambda$. We shall apply this construction to radical square zero tame hereditary algebras. By definition, such an algebra is of the form
$kQ$, where the quiver $Q$ is a Euclidean diagram of type $\tilde{A}_n, \tilde{D}_n$ or $\tilde{E}_{6,7,8}$ with an orientation, such that there are no paths of length
$2$. (For $\tilde{A}_n$ this is possible if only if $n$ is odd.) 

We denote by $\cZ$ the center of the algebraic group scheme $\SL(2)$. Given a finite subgroup scheme $\cG \subseteq \SL(2)$, we put $\PP(\cG) := \cG/(\cG\cap\cZ)$. The unique largest linearly reductive normal subgroup scheme of a given finite group scheme $\cG$ will be denoted $\cG_{\rm lr}$.

Our final result characterizes the representation-infinite finite algebraic groups of domestic representation type, showing their close connection with binary polyhedral groups and tame hereditary algebras. Recall that a finite linearly reductive subgroup scheme of $\SL(2)$ is referred to as a {\it binary polyhedral group scheme}. We denote by $\SL(2)_1$ the first Frobenius kernel of $\SL(2)$.

\bigskip 

\begin{Thm} \label{GS3} Let $\cG$ be a finite group scheme of characteristic $p\ge 3$ over $k$. Then the following statements are equivalent:
\begin{enumerate}
\item The principal block $\cB_0(\cG)$ is representation-infinite and of polynomial growth.
\item There exists a radical square zero tame hereditary algebra $\Lambda$ such that $\cB_0(\cG)$ is Morita equivalent to $T(\Lambda)$.
\item The Hopf algebra $k\cG$ is representation-infinite and domestic.
\item The principal block $\cB_0(\cG)$ is representation-infinite and domestic. 
\item There exists a binary polyhedral group scheme $\tilde{\cG} \subseteq \SL(2)$ such that $\cG/\cG_{\rm lr} \cong \PP(\SL(2)_1\tilde{\cG})$. \end{enumerate} \end{Thm}

\begin{proof} Let $G := \cG(k)$ be the finite group of $k$-rational points. By general theory, we have
\[ k\cG = k\cG^0\sharp kG = k\cG^0\!\ast\!G.\]
(1) $\Rightarrow$ (2) In view of the above, Lemma \ref{GS2} implies that the extension $\cB_0(\cG)\!:\!\cB_0(\cG^0)$ is split. Theorem \ref{SpE3} thus shows that the algebra $\cB_0(\cG^0)$ is of polynomial growth. Thanks to \cite[(6.1)]{FS}, $\cB_0(\cG^0)$ is domestic. By \cite[(3.1)]{FV}, $\cB_0(\cG^0)$ is not representation-finite and \cite[(5.1)]{FS} implies that the center ${\rm Cent}(\cG^0)$ is diagonalizable. The assertion now follows from \cite[(7.3.1)]{Fa2}. 

(2) $\Rightarrow$ (3) The algebra $\Lambda$ is known to be representation-infinite and domestic \cite{DF,Na}. Thanks to \cite{Ta}, the trivial extension $T(\Lambda)$ also enjoys this property. It follows that the algebra $\cB_0(\cG)$ is domestic and of infinite representation type. As before, we see that $\cB_0(\cG^0)$ also has these properties, and \cite[(6.1)]{FS} implies the domesticity of $k\cG^0$. Since $\cB_0(\cG)$ is not representation-finite, \cite[(6.2.1)]{Fa2} ensures that $p$ does not divide ${\rm ord}(G)$. It now follows from Corollary \ref{SP3}(2) that $k\cG = k\cG^0\!\ast\! G$ is of polynomial growth with rate $\gamma_{k\cG} = \gamma_{k\cG^0} \le 1$. Hence $k\cG$ is domestic and of infinite representation type.

(3) $\Rightarrow$ (4) This follows from \cite[(3.1)]{FV}.

(4) $\Rightarrow$ (5) Let $T\subseteq \SL(2)$ be the standard torus of diagonal $(2\times 2)$-matrices of determinant $1$ and put $\cG' := \cG/\cG_{\rm lr}$. Then we have $\cG'_{\rm lr} = e_k$, while \cite[(1.1)]{Fa2} implies that $\cB_0(\cG') \cong \cB_0(\cG)$ is domestic. Since $\cG'_{\rm lr}=e_k$, a consecutive application of \cite[(1.2)]{Fa2} and \cite[(6.1)]{FS} provides an isomorphism $\cG'^0 \cong \SL(2)_1T_r$ for some $r\ge 1$. According to \cite[(6.2.1)]{Fa2}, the prime $p$ does not divide ${\rm ord}(\cG'(k))$, so that \cite[(1.2)]{Fa2} yields $C_{\cG'} := {\rm Cent}_{\cG'_{\rm red}}(\cG'^0)=e_k$. Consequently, ${\rm Cent}_{\cG'}(\cG'^0) = e_k$, and our assertion follows from \cite[(7.1.2)]{Fa2}.  

(5) $\Rightarrow$ (1) Let $\tilde{\cG} \subseteq \SL(2)$ be a binary polyhedral group scheme. We consider the semidirect product
\[ \cG := \SL(2)_1\rtimes \tilde{\cG}.\]
Since $\tilde{\cG}$ is linearly reductive, Corollary \ref{SP4} guarantees that $k\cG \cong k\SL(2)_1\sharp k\tilde{\cG}$ is domestic. It follows that the factor algebra $k\SL(2)_1\tilde{\cG}$
of $k\cG$ is domestic. Consequently, $\cB_0(\PP(\SL(2)_1\tilde{\cG})) \cong \cB_0(\SL(2)_1\tilde{\cG})$ enjoys the same property. Let $\cx_\cG(k)$ be the {\it complexity} of the trivial $\cG$-module, that is, the polynomial rate of growth of a minimal projective resolution of $k$. If $\cB_0(\SL(2)_1\tilde{\cG})$ is  representation-finite, then $1 \ge \cx_{\SL(2)_1\tilde{\cG}}(k) \ge \cx_{\SL(2)_1}(k)=2$, a contradiction. As a result, the principal block $\cB_0(\cG) \cong \cB_0(\cG/\cG_{\rm lr}) \cong \cB_0(\PP(\SL(2)_1\tilde{\cG}))$ (cf.\ \cite[(1.1)]{Fa2}) is of polynomial growth. \end{proof}

\bigskip

\begin{Remark} The binary polyhedral group schemes are completely understood, see \cite[(3.3)]{Fa2}. They are determined  
by their McKay quivers, which in turn give rise to the $\Ext$-quivers of the principal blocks $\cB_0(\PP(\SL(2)_1\tilde{\cG})$. \end{Remark}

\end{document}